\documentclass[12pt]{amsart}
\usepackage{epsfig,color}
\usepackage{amssymb} 

\makeatother

\theoremstyle{definition}

\def\fnum{equation}
\newtheorem{Thm}[\fnum]{Theorem}
\newtheorem{Cor}[\fnum]{Corollary}

\newtheorem{Lem}[\fnum]{Lemma}

\newtheorem{Pro}[\fnum]{Proposition}

\numberwithin{equation}{section}

\newcommand{\Vol}{{\text{Vol}}}

\newcommand{\nn}{{\bf{n}}}
\newcommand{\tn}{{\tilde{\bf{n}}}}
\newcommand{\Ric}{{\text{Ric}}}

\newcommand{\Hess}{{\text {Hess}}}

\def\RR{{\bold R}}

\def\SS{{\bold S}}

\newcommand{\dv}{{\text {div}}}
\newcommand{\e}{{\text {e}}}

\newcommand{\R}{{\text{R}}}

\newcommand{\eqr}[1]{(\ref{#1})}

\def\bB{{\bold B}}

\title{Gradient estimates for scalar curvature}
\author[]{Tobias Holck Colding}%
\address{MIT, Dept. of Math.\\
77 Massachusetts Avenue, Cambridge, MA 02139-4307.}
\author[]{William P. Minicozzi II}%
 
\thanks{The  authors
were partially supported by NSF  DMS Grants   2405393 and 2304684.}

\email{colding@math.mit.edu  and minicozz@math.mit.edu}

\begin{document}

\maketitle

\begin{abstract}
A gradient estimate is a crucial tool used to control the rate of change of a function on a manifold,  paving the way  for deeper analysis of geometric properties.   A celebrated result of Cheng and Yau gives gradient bounds on manifolds with Ricci curvature $\geq 0$.  
The Cheng-Yau bound is not sharp, but there is a sharp gradient  estimate.   To explain this,  a Green's function $u$  on a manifold  can be used to define a regularized distance  $b= u^{\frac{1}{2-n}}$ to the pole.  
On $\RR^n$, the level sets of $b$ are spheres and $|\nabla b|=1$.
If $\Ric \geq 0$, then \cite{C3} proved 
 the sharp gradient estimate $|\nabla b| \leq 1$.
We show that the average of $|\nabla b|$ is $\leq 1$ 
on a three manifold with  nonnegative scalar curvature.  The average is over any  level set  of $b$ and  if  the average is one on even one level set, then  $M=\RR^3$.    
 \end{abstract}

\section{Introduction}
Gradient estimates for harmonic functions are ubiquitous in geometric analysis.   They are typically proven using the maximum principle
and require a lower Ricci curvature bound.   For scalar curvature, those techniques are not available and one does not expect a pointwise gradient bound.  Instead we will use different techniques to prove a sharp average gradient estimate in three dimensions.

Suppose that $M^3$ is an open three-manifold,  $u>0$ is a   Green's function with a pole at   $p \in M$
 and
 $u(x) , |\nabla u|(x) \to 0$ as $x \to \infty$.{\footnote{When $\inf_M \, \Ric > - \infty$ and   $u \to 0$ at infinity, then   $|\nabla u| \to 0$ at infinity by \cite{ChY}.}}
   We use the normalization 
  $\Delta \, u = -4\, \pi \,  \delta_p$.  The function $b = u^{-1}$ vanishes at $p$, is positive away from $p$, and  gives a   ``regularized distance.''
    On $\RR^3$, we have
$u(x) = |x-p|^{-1}$,  $b(x) = |x-p|$ is a distance function, and 
  $|\nabla b| $ is one.    

The celebrated Cheng-Yau gradient estimate bounds  $\frac{\nabla u}{u}$ in terms of   the distance to $p$ if $\Ric \geq 0$, \cite{ChY,Y1}. 
This 
  implies
   that $|\nabla b| \leq C$ for some  constant $C> 1$.   The sharp gradient estimate in \cite{C3} gives the stronger estimate
   $|\nabla b| \leq 1$,
   with equality at one point if and only if $M$ is Euclidean.   Both \cite{ChY} and \cite{C3} hold in all dimensions with obvious 
   modifications\footnote{In higher dimensions, $\Delta \, u = - \Vol (\SS^{n-1}) \, \delta_p$ and $b$ is given by  $b^{2-n}=u$; see \cite{C3}.}.
   
   This sharp gradient estimate is closely related to some monotonicity formulas,  \cite{C3}.  To explain this, we
follow \cite{CM4} and
define   $A_{\beta} (r) $   by
\begin{align}
	A_{\beta} (r) &= r^{-2}\int_{b=r}|\nabla b|^{1+\beta} \,  .	\label{e:Abeta}
\end{align}
If  $\beta = 0$, then 
$A_0 (r)= 4\, \pi$ is constant by the divergence theorem.  
Thus, if 
 $M^3$ has $\Ric \geq 0$ (and similarly in all dimensions), then  the sharp gradient estimate gives a bound for $A_1(r)$
\begin{align}
	A_1 (r) = r^{-2} \,  \int_{b=r} |\nabla b|^2 \leq \left( \sup |\nabla b| \right) \, A_0 (r) = 4\, \pi \, \sup |\nabla b|  \leq 4 \, \pi \, ,
\end{align}
 with equality  only when $M$ is Euclidean space. 
The  bound $A_1 (r) \leq 4 \, \pi$ 
says that $|\nabla b| \leq 1$ in an average sense.   We will prove this for $S \geq 0$ in three dimensions:

 \begin{Thm}	\label{t:main1}
If $M^3$ has  first betti number zero,  one end, and  $S \geq 0$, then $A_1 (r) \leq 4 \, \pi$ with equality for some $r$  if and only if $M = \RR^3$.
\end{Thm}
 
  Corollary $1.2$ in \cite{MW1} proved that $A_1 (r) \leq 4\, \pi$ in the special case where $M$ has asymptotically nonnegative Ricci curvature.

 The main ingredient in Theorem \ref{t:main1} is the following quadratic lower bound for the growth of $A_1$   (this does   not  use
  that $|\nabla u| \to 0$).

\begin{Pro}	\label{t:cgrow}
Suppose that  $M^3$ has  first betti number zero,  one end, and  $S \geq 0$.  If   $A_1(r_0)>4\,\pi$, then there exists $c>0$ so that for $r\geq r_0$
\begin{align}  \label{e:te2}
A_1(r)&\geq c\,r^{2}\,  .
\end{align}
\end{Pro}

Since $|\nabla u| = b^{-2} \, |\nabla b|$, 
quadratic growth of $A_1$  is   a positive lower bound for 
the weighted 
  average of $|\nabla u|$ near infinity, contradicting that $|\nabla u| \to 0$.


\vskip1mm
The next corollary gives a Cohn-Vossen type bound for the  average of the scalar curvature that is related to a conjecture of Yau, \cite{Y2}. See  \cite{X,Z}
 for 
 versions of this for $\Ric \geq 0$.
 
 \begin{Cor}	\label{t:main1S}
If $M^3$ has  first betti number zero,  one end, and  $S \geq 0$, then  for any $r>0$ 
\begin{align}
	r^{-1} \, \int_0^r \int_{b=s} S \leq 48\, \pi 
	\, .
\end{align}
\end{Cor}

  \vskip2mm
  Gradient estimates are closely connected to monotonicity formulas involving the Green's function, starting with three monotonicity formulas in \cite{C3}.
  The  level sets of $b$ are analogous to distance spheres, but they come with a natural  measure $|\nabla u|$ that is preserved.{\footnote{This works in  general dimensions; see  \cite{ChC1, ChC2, C1,C2,CM1,CM2}.}}   \cite{CM4} showed that  the monotonicities   were contained in   one-parameter families  and \cite{CM5} used these quantities to understand the asymptotic structure of certain Einstein manifolds.  There are now many applications along these lines, including    \cite{AFM1,AFM2,  BFM, BS, GV, HP, P}, cf. also \cite{BR}.
   In the last several years, there have been a number of monotonicity formulas using $A_1 (r)$ in three dimensions with lower bounds on scalar curvature.  This started in the papers of Munteanu-Wang, \cite{MW1, MW2}, on nonnegative scalar curvature, 
 the work of  Agostiniani-Mazzieri-Oronzio, \cite{AMO}, on the positive mass theorem (cf. Bray-Kazaras-Khuri-Stern, \cite{BKKS},  Xia-Yin-Zhou, \cite{XYZ}), the work of Agostiniani-Mantegazza-Mazzieri-Oronzio, \cite{AMMO}, on the Penrose inequality, and also appears in 
  the work of Chodosh-Li, \cite{ChL}, on stable minimal $3$-manifolds and the Bernstein problem.

\section{The differential inequality for $A_1$}

In this section, we will prove a differential inequality for $A_1(r)$ in terms of several natural integrals on the level sets of $b$.
Let $\Sigma$ denote a level set  of $b$ and  $K$ the curvature of $\Sigma$. The outward unit normal $\nn$ is given by
\begin{align}	
\nn &= \frac{\nabla b}{|\nabla b|} \, .
\end{align}
The Laplacian of $b^2$  and   trace-free part $\bB$  of the hessian of $b^2$  are  
\begin{align}	\label{e:deltab2}
	\Delta\,b^2&=6\,|\nabla b|^2\,   , \\
	\bB&=\Hess_{b^2}-2\,|\nabla b|^2\,g\,  .  \label{e:Bbb}
\end{align}
On $\RR^3$, 
  $u = |x|^{-1}$,  $b^2=|x|^2$  and, thus, 
 $\bB$ vanishes.  This reflects the conical structure of $\RR^3$.

Define   $B_1, B_2, S_1$ by
\begin{align}
B_1(r) &= r^{-2} \, \int_{b=r} \Hess_{b^2}(\nn , \nn) \, , \\
	B_2 (r) &=   \int_{b=r} \frac{ |\bB|^2}{|\nabla b^2|^2}  \, , \\
	S_1 (r) &= \int_{b=r} S \, .
\end{align}
The function $A_1$ is locally Lipschitz, differentiable almost everywhere and absolutely continuous,  while $B_1$ and $S_1$ are continuous   (see Appendix \ref{s:app2}).

The next proposition  implies that  $r\, A_1'(r) \geq A_1 (r) - 4\, \pi$ when $S \geq 0$, which can be seen to be  equivalent to the monotonicity formula in 
\cite{MW1}.{\footnote{Theorem $1.1$ in \cite{MW1} states that $\frac{d}{dt} \, \left( \frac{1}{t} \, \int_{u=t} |\nabla u|^2 \right) \leq 4 \, \pi$.  If $r = \frac{1}{t}$, this is
$r\, A_1'(r) \geq A_1 (r) - 4\, \pi$.}}

\begin{Pro}	\label{c:MWa}
If  the level sets of $b$ are connected, then
	\begin{align}	
			r\, A_1'(r) &\geq A_1 (r) - 4\, \pi + \frac{1}{2\, r} \, \int_0^r \left( S_1 (s) + B_2 (s)\right)  \, ds \, .
	\end{align}
\end{Pro}

\subsection{Differentiating $A_1$}

 The next lemma recalls the formula for the derivative of $A_1$; this is contained in  \cite{C3,CM4}, but we include the proof for
 completeness.  This calculation uses that $u$ is a positive proper harmonic function with precise asymptotics at the pole $p$, but it does not require curvature or topological assumptions on $M$.

\begin{Lem}	\label{l:A1p}
The function $A_1$ is continuously differentiable with
\begin{align}
	r\, A_1'(r)&=  \frac{1}{2} \, B_1 (r) - A_1(r) = \frac{1}{2} \, r^{-2} \, \int_{b=r} \bB (\nn , \nn)  \,  , \label{e:claim1A}  \\
		2\,  \left( r \, A_1(r) \right)' &= B_1(r)  \, . \label{e:claim1B} 
\end{align}
\end{Lem}

\begin{proof}
Since $\frac{|\nabla b|}{b} = |\nabla \log u|$, the asymptotics of the Green's function   near the pole $p$ (lemma $2.1$ in \cite{C3}, \cite{GS})
gives that  
\begin{align}
	  \lim_{s\to 0} \, \int_{b=s} b^{-2} \,  |\nabla b|^2 = 4\, \pi  \, .
\end{align}
Using that $\dv (b^{-2} \, \nabla b) = 0$ away from $p$ and
\begin{align}
	\nabla \, |\nabla b| = \nabla \, \left(  \frac{ |\nabla b^2|}{2b}
	\right) = \frac{ \Hess_{b^2} (\nn , \cdot)}{2b} - \frac{ |\nabla b^2| \, \nabla b}{2b^2} \, ,
\end{align}
the divergence theorem gives that
\begin{align}
	A_1(r) &= \int_{b=r} \langle |\nabla b| \, (b^{-2} \, \nabla b) , \nn \rangle =
	4\, \pi  +  \int_{b<r} b^{-2} \, \langle \nabla b , \frac{ \Hess_{b^2} (\nn , \cdot)}{2b} - \frac{ |\nabla b^2| \, \nabla b}{2b^2} \rangle \notag \\
	&=4\, \pi + \int_{b<r} b^{-3} \,  \left( \frac{1}{2} \, |\nabla b| \,   \Hess_{b^2} (\nn , \nn) -  |\nabla b|^3 \right)    
	\,  ,
\end{align}
where we also used the asymptotics at $b=0$ to evaluate the inner boundary term.
Combining  this and the coarea formula gives
\begin{align}	\label{e:thihere}
	r\, A_1'(r)&= r^{-2} \, \int_{b=r}
	\left( \frac{1}{2} \,   \Hess_{b^2} (\nn , \nn) -  |\nabla b|^2 \right) = \frac{1}{2} \, B_1 (r) - A_1(r)    \,  .
\end{align}
Note that the right-hand side is continuous in $s$ (by Appendix \ref{s:app2}); since $A_1$ is absolutely continuous, we see that $A_1$ is continuously differentiable.  
Furthermore, \eqr{e:thihere}  gives the first equality in \eqr{e:claim1A}. To get the second equality, use \eqr{e:Bbb} to get that
\begin{align}	 
	\frac{1}{2} \, \int_{b=r}  \bB (\nn , \nn) = \int_{b=r} \left( \frac{1}{2} \, \Hess_{b^2} (\nn , \nn) - |\nabla b|^2 \right)  \,  .
\end{align}
  Finally, the last claim \eqr{e:claim1B} follows easily from the first equality in \eqr{e:claim1A}.
\end{proof}

\subsection{Level set geometry}

The next proposition expresses $\Delta |\nabla b^2|$ in terms of the scalar curvature of $M$ and the geometry of the level set.
This result is local and does not require any assumptions on $M$ (cf. equation $(1.2)$ in \cite{MW2}).

\begin{Pro}	\label{l:sternbochA}
We have that
\begin{align}
	  \frac{\Delta \, |\nabla b^2|}{|\nabla b^2|} & =      \frac{1}{2} \, S -   K +  \frac{ |\bB|^2}{ 2\, |\nabla b^2|^2}  +\frac{3}{2} \, \frac{ \Hess_{b^2}(\nn , \nn)  }{  b^2 }
	\, . \notag
\end{align}
\end{Pro}

\vskip2mm
 The next lemma uses the gauss equation (cf. \cite{SY,JK,S,BKKS,AMO}) to express the Ricci curvature in the direction normal to a level set  in terms of the ambient scalar curvature, the gauss curvature of the level set, and  derivatives of  the function and then combines this with the Bochner formula.

 \begin{Lem}	\label{l:sternboch}
If $v$ is a function on $M^3$ and $\tn = \frac{\nabla v}{|\nabla v|}$, then
\begin{align}
	\frac{\Delta \, |\nabla v|}{|\nabla v|} &= \frac{1}{2} \, S - K + \frac{1}{2} \, \frac{ |\Hess_v|^2}{|\nabla v|^2} + \frac{1}{2} \, |\nabla v|^{-2}\,\left[|\Delta\,v|^2-2\,\Delta\,v\,\Hess_v( \tn ,  \tn)\right]+\frac{\langle\nabla \Delta\,v, \tn\rangle}{|\nabla v|}\, .\notag
\end{align}
\end{Lem}

\begin{proof}
Let $e_1 , e_2$ be a local orthonormal frame for the level sets of $v$, so that $e_1, e_2, \tn$ is a local orthonormal frame for $M^3$.
  The definitions of $S$ and $\Ric$ give
\begin{align}	\label{e:SY1}
	S = \Ric (e_1 , e_1) + \Ric (e_2 , e_2) + \Ric (\tn , \tn)   = 2 \, \R(e_1 , e_2 , e_1 , e_2) + 2 \, \Ric (\tn , \tn) \, .
\end{align}
The Gauss equation   gives that 
\begin{align}
	K &= \R(e_1 , e_2 , e_1 , e_2) + \frac{1}{2} \,   \left| H_{\Sigma} \right|^2
	- \frac{1}{2} \, |A_{\Sigma}|^2  \, .	\label{e:SY2}
\end{align}
Combining the last two equations, we see that
\begin{align}	\label{e:att}
	S &=2 \, \Ric (\tn , \tn) +  2\, K +  |A_{\Sigma}|^2 - |H_{\Sigma}|^2 \, .
\end{align}
We also have that
\begin{align}   \label{e:AfromHess}
	|\nabla v|^2 \, |A_{\Sigma}|^2 = \left| \Hess_v \big|_{\Sigma} \right|^2 = |\Hess_v|^2 +  (\Hess_v ( \tn , \tn))^2 - 2 \, 
	|\nabla |\nabla v||^2 \, .
\end{align}
 Using this in \eqr{e:att} gives that  (cf. \cite{S, SY}) 
\begin{align}
	S&= 2 \, \Ric (\tn , \tn) +  2\, K - |H_{\Sigma}|^2 + 
	|\nabla v|^{-2} \, \left(    |\Hess_v|^2 +  (\Hess_v( \tn , \tn))^2 - 2 \, 
	|\nabla |\nabla v||^2 \right)  \, .
\end{align}
Using that $|\nabla v| \, H_{\Sigma} = |\nabla v| \,  \dv \frac{\nabla v}{|\nabla v|} =   \Delta v - \Hess_v (\tn , \tn) $     gives 
\begin{align}
	\Ric (\tn ,  \tn) =  \frac{1}{2} \, S  -  K  +
	|\nabla v|^{-2}\,\left[|\nabla |\nabla v||^2+ \frac{1}{2} \, |\Delta\,v|^2- \Delta\,v\,\Hess_v(\tn,\tn)-\frac{1}{2}\,|\Hess_v|^2\right]\, .
\notag 
\end{align}
By the Bochner formula
\begin{align}
	 |\nabla v| \, \Delta \, |\nabla v| + |\nabla |\nabla v||^2 = \frac{1}{2} \, \Delta |\nabla v|^2 = |\Hess_v|^2 +\langle \nabla\Delta\,v,\nabla v\rangle+ \Ric (\nabla v , \nabla v)\,  ,
\end{align}
which can be re-written as
\begin{align}
\frac{\Delta\,|\nabla v|}{|\nabla v|}-\frac{|\Hess_v|^2}{|\nabla v|^2}+\frac{|\nabla |\nabla v||^2}{|\nabla v|^2}-\frac{\langle \nabla \Delta\,v,\nn\rangle}{|\nabla v|}=\Ric(\tn,\tn)\,  .
\end{align}
Combining this with the formula for $\Ric (\tn ,  \tn)$  gives the claim.  
\end{proof}

Specializing   to the function $v=b^2$ gives Proposition \ref{l:sternbochA}:

\begin{proof}[Proof of Proposition \ref{l:sternbochA}]
Since $\Delta \, b^2 = 6 \, |\nabla b|^2$, Lemma \ref{l:sternboch} gives
\begin{align}
	\frac{\Delta \, |\nabla b^2|}{|\nabla b^2|} &= \frac{1}{2} \, S - K + \frac{1}{2} \, \frac{ |\Hess_{b^2}|^2}{|\nabla b^2|^2} + \frac{1}{2} \, |\nabla b^2|^{-2}\,\left[|\Delta\,b^2|^2-2\,\Delta\,b^2\,\Hess_{b^2}(\nn,\nn)\right]+\frac{\langle\nabla \Delta\,b^2,\nn\rangle}{|\nabla b^2|} \notag \\
	&= \frac{1}{2} \, S - K +  \frac{ |\Hess_{b^2}|^2}{ 8\, b^2 \, |\nabla b|^2} + \frac{\left[ 9\, |\nabla b|^4-3\, |\nabla b|^2\,\Hess_{b^2}(\nn,\nn)\right]}{ 2\, b^2 \, |\nabla b|^{2}}+3\, \frac{\langle \nabla |\nabla b|^2,\nn\rangle}{  b \, |\nabla b|}
\end{align}
Since
 $
\nabla b^2=2\,b\,\nabla b$, we have that
$|\nabla b^2|^2=4\,b^2\,|\nabla b|^2$ and, thus, 
\begin{align}	\label{e:calc1A}
 \langle \nabla |\nabla b|^2,\nn\rangle = b^{-1} \, |\nabla b| \, \Hess_{b^2}(\nn , \nn) - 2 \, b^{-1} \, |\nabla b|^3
 = \frac{|\nabla b|}{b} \, \bB (\nn , \nn) \,  .
\end{align}
Therefore, we get that
\begin{align}
	2\, \frac{\Delta \, |\nabla b^2|}{|\nabla b^2|} & =  S - 2\, K +  \frac{ |\Hess_{b^2}|^2}{ 4\, b^2 \, |\nabla b|^2} + \frac{  9\, |\nabla b|^4-3\, |\nabla b|^2\,\Hess_{b^2}(\nn,\nn) }{ b^2 \, |\nabla b|^{2}}\notag \\
	&+6\, \frac{b^{-1} \, |\nabla b| \, \Hess_{b^2}(\nn , \nn) - 2 \, b^{-1} \, |\nabla b|^3 }{  b \, |\nabla b|}  \\
	&=   S - 2\, K +  \frac{ |\Hess_{b^2}|^2 - 12 \, |\nabla b|^4}{ 4\, b^2 \, |\nabla b|^2}  +3\, \frac{ \Hess_{b^2}(\nn , \nn)  }{  b^2  }
	\, . \notag
\end{align}
The corollary follows from this and 
\begin{align}
|\bB|^2=|\Hess_{b^2}|^2-12\,|\nabla b|^4\,  ,	\label{e:Bsq}
\end{align}
which uses \eqr{e:deltab2}.
\end{proof}

 \subsection{A lower bound for $B_1$}
 
 We turn next to  a lower bound for  $B_1$.  This requires  topological control on the level sets so that the Gauss-Bonnet theorem yields an integral bound on the curvature
 (the level sets are surfaces so Gauss-Bonnet applies). The local calculations of the previous subsection will be used.

 \begin{Pro}	\label{p:insert}
 If the level sets of $b$ are connected, then for $0< r_1 < r_2 < \infty$ regular values
\begin{align}
	r\, B_1 (r)  &   \geq  4\, r \, A_1 (r) - 8 \, \pi \, r + \int_0^{r} \left( S_1 (s)+ B_2 (s)   
\right)\, ds \, .
\end{align}
 \end{Pro}
 
 \begin{proof}
Combining  the fact that $\langle \nabla |\nabla b^2|,\nn\rangle=\Hess_{b^2}(\nn,\nn)$, the divergence theorem and the coarea formula for $b$  gives that
\begin{align}	 
	s\, B_1 (s) & = s^{-1} \, \int_{b=s}\Hess_{b^2}(\nn,\nn) =  \int_{b=s} \langle b^{-1} \, \nabla |\nabla b^2|,\nn\rangle \notag \\
	&= \int_{b \leq s} \,
	\left(  \frac{\Delta \, |\nabla b^2|}{b} -  \frac{\langle \nabla b , \nabla |\nabla b^2| \rangle }{b^2} \right) 
	=  \int_0^{s}\int_{b=t}\left(
  \frac{\Delta \, |\nabla b^2|}{b\, |\nabla b|} - \frac{\langle \nn , \nabla |\nabla b^2| \rangle }{b^2} 
 \right) \, ,
\end{align}
where the third equality used the asymptotics at the pole{\footnote{Combining the \cite{GS} first order asymptotics of $u$ at the pole with standard elliptic estimates for the harmonic function $u$ on a ball of size $b/2$ gives that $\Hess_{b^2} (\nn , \nn)$ is bounded near $p$.}}
 to see that the inner boundary term goes to zero.
Using Proposition \ref{l:sternbochA}, the Gauss-Bonnet theorem and the connectedness of the level sets,  we get that
\begin{align}
s\, B_1 (s)  &= 
\int_0^{s} \left( S_1 (t)+ B_2 (t)  + 2\, B_1 (t) - 2\, \int_{b=t} K 
\right) \notag \\
&  \geq  \int_0^{s} \left( S_1 (t)+ B_2 (t)  + 2\, B_1 (t) -  8\, \pi 
\right)
 \, .	\label{e:e130}
\end{align}
Lemma \ref{l:A1p}
gives that $r\, A_1'(r)=  \frac{1}{2} \, B_1 (r) - A_1(r) $, so  that $\int_0^r B_1 (s) = 2 \, r \, A_1 (r)$.  Using this in \eqr{e:e130}, we conclude that
\begin{align}
	r\, B_1 (r)  &   \geq  4\, r \, A_1 (r) - 8 \, \pi \, r + \int_0^{r} \left( S_1 (s)+ B_2 (s)   
\right)\, ds \, .
\end{align}

\end{proof}

\begin{proof}[Proof of Proposition \ref{c:MWa}]
Since $r \, A_1'(r) = \frac{1}{2} \, B_1(r) - A_1(r)$  by Lemma \ref{l:A1p}, Proposition \ref{p:insert} gives that
\begin{align}
	2\, r\, A_1'(r) =  B_1(r) - 2\, A_1(r) \geq  2\,  A_1 (r) - 8 \, \pi   + \frac{1}{r} \, \int_0^{r} \left( S_1 (s)+ B_2 (s)   \right) \, ds \, .
\end{align}
\end{proof}

\section{The  main theorems} 
  
  In this section, we will study the growth $A_1(r)$ and use this to prove the main theorems.  The starting point is the differential inequality for $A_1$
  from Proposition \ref{c:MWa}.  When $S \geq 0$,  this gives the differential inequality  $r\, A_1'(r)  \geq A_1 (r) - 4\, \pi$ which leads to linear growth of $A_1$ if it gets above $4\, \pi$.  To get to quadratic growth, we will need  the   $B_2(r)$ term. 
 This is non-standard since $B_2$ vanishes in the model case.
 
 \vskip1mm
  The next lemma gives a lower bound for $B_2$ in terms of $A_1$ and $A_1'$.
  
  \begin{Lem}	\label{c:csCS}
We have that
\begin{align}
	(r\,A_1')^2&\leq  \frac{2}{3} \, A_1\,B_2 \, .
\end{align}
\end{Lem}

\vskip1mm
The proof uses an elementary linear algebra fact (that is usually used for the  Kato inequality for harmonic functions, see, e.g., lemma $2.2$ in \cite{MW1}, but is used differently here):

\begin{Lem}	\label{l:CS}
If $C$ is a trace-free three-by-three  matrix and $v$ is a unit vector, then
\begin{align}
	|C(v,v)|^2\leq |C(v)|^2  \leq \frac{2}{3} \, |C|^2 \, .
\end{align}
\end{Lem}

\begin{proof}
Choose an orthonormal basis $e_1 , e_2 , e_3$  so that  $e_1 = v$.  It follows that
\begin{align}
	 |C(v,v)|^2 =C_{11}^2    \leq  C_{11}^2 + C_{12}^2 + C_{13}^2 = |C(v)|^2 \, .
\end{align}
Expanding $|C|^2$ and using that $C_{33} = - (C_{11} + C_{22})$, we have
\begin{align}
	|C|^2 &=  C_{11}^2 + C_{22}^2 + C_{33}^2 + \sum_{i \ne j} C_{ij}^2 = C_{11}^2 + C_{22}^2 + (C_{11} + C_{22})^2 
	+ 2\, \sum_{i < j} C_{ij}^2  \notag \\
	&= 2\, C_{11}^2 + 2\, C_{22}^2 + 2 \, C_{11} \, C_{22} + 2\, (C_{12}^2 + C_{13}^2+ C_{23}^2) \\
	&= \frac{3}{2} \, C_{11}^2 +  2\, (C_{12}^2 + C_{13}^2)
	+ \left( \sqrt{2} \, C_{22} + \frac{1}{\sqrt{2}} \, C_{11} 
	\right)^2  + 2\,   C_{23}^2  \, . 
	\notag
\end{align}
\end{proof}

\vskip1mm
The lemma is sharp on the three-by-three diagonal matrix with entries $2$, $-1$ and $-1$.

\begin{proof}[Proof of Lemma \ref{c:csCS}]
To see this, use the first claim in Lemma \ref{l:A1p},  the Cauchy-Schwarz inequality and then  Lemma \ref{l:CS} to get that
\begin{align}   \label{e:cscrucial}
  4 \, (r\, A_1'(r))^2 
  &= \left(r^{-2}\int_{b=r} \bB(\nn,\nn)\right)^2 \notag \\
  &\leq  \left( r^{-2} \, \int_{b=r} |\nabla b|^2\right) \, 
  \left( 	r^{-2} \, \int_{b=r} \frac{[ \bB(\nn,\nn)]^2}{|\nabla b|^2}
  \right) 
   \\
  &\leq  A_1 (r) \,\left( \frac{2}{3} \int_{b=r}\frac{|\bB|^2}{|\nabla b|^2}\right) = \frac{8}{3} \, A_1 (r) \,  B_2 (r) \,  . \notag
\end{align}

\end{proof}

\subsection{The growth of $A_1$}
We will now turn to proving a lower bound for the growth of $A_1(r)$ when $S \geq 0$, 
  the level sets of $b$ are connected, and there exist $\delta > 0$ and  $r_0 > 0$ so that 
\begin{align}
	A_1 (r_0) > 4 \, \pi + \delta \, .
\end{align}
Proposition \ref{c:MWa} gives that   $A_1$ is increasing for $r\geq r_0$ with
 \begin{align}
 	(\log (A_1(r)- 4\, \pi))' \geq r^{-1} \, , 
\end{align}
 so that integrating this from $r_0$ to $r$ gives that
 \begin{align}	\label{e:lowerA}
 	A_1(r) \geq 4\, \pi + \delta \, \frac{r}{r_0} \, .
 \end{align}
 Define a continuous positive function
  \begin{align}
 	a(r) = r \, \left( \log A_1 \right) '(r) 
\end{align}
that measures the rate of polynomial growth of $A_1(r)$.  To get that $A_1$ grows quadratically, we will show that $a(r)$ rapidly approaches 
$2$ as $r$ grows.

\vskip2mm
The next proposition says that the continuous function $a(r)$   satisfies the differential inequality
\begin{align}
	r\, a' \geq  1 - a^2/4 - 4\, \pi/A_1  	\label{e:dia}
\end{align}
in an integral sense.

 \begin{Pro}	\label{p:aprime}
 If $r_1 < r_2$ are regular values of $b$ with $r_0 \leq r_1$, then 
 \begin{align}
 	a(r_2) - a(r_1) \geq  \int_{r_1}^{r_2} \left( 1 - \frac{a^2(r)}{4} - \frac{4\, \pi}{A_1 (r)} 
	\right) \, \frac{dr}{r} \, .
 \end{align}
 \end{Pro}
 
 \vskip1mm
 The proposition will follow from applying the divergence theorem and co-area formula to a well-chosen vector field $V$ that is defined in the next lemma.
 
 \begin{Lem}	\label{l:defineVforme}
If we define a $C^1$  vector field 
$V= b^{-2} \,  \frac{ \nabla |\nabla b^2|}{A_1 (b)} $, then at a regular value $r$
  \begin{align}	
	     \int_{b=r} \langle V , \nn \rangle & = 2 \, a(r) + 2  \, , \label{e:defV} \\
	r\, \int_{b=r} \frac{\dv \, V}{|\nabla b|} &\geq   
	 2  - \frac{1}{2} \, a^2(r)
	 - \frac{8\, \pi }{A_1(r)} 
	\, . \label{e:usethisV}
\end{align}
  \end{Lem}
 
 \begin{proof}
 Using the definitions of $V$ and $B_1$, we see that 
  \begin{align}
 	\int_{b=r} \langle V , \nn \rangle = \frac{1}{A_1 (r)} \, r^{-2} \, \int_{b=r} \langle \nabla |\nabla b^2| , \nn \rangle = \frac{B_1(r)}{A_1(r)} \, .
	\label{e:Vnonr}
 \end{align}
 Therefore,   Lemma \ref{l:A1p}
gives for each $r$ that
  \begin{align}	
	2\, a(r) + 2&= \frac{B_1 (r)}{A_1 (r)} =    \int_{b=r} \langle V , \nn \rangle \, ,  
\end{align}
giving the first claim.

 We turn now to the second claim.
 Define a $C^1$ function   $v   = \frac{1}{A_1 (b)}$, so that
$V= b^{-2} \, v \,  \nabla |\nabla b^2|$.
 The product rule and the chain rule give that
 \begin{align}
	\frac{\dv \, V}{|\nabla b|} &= b^{-2} \, v \,  \frac{\Delta |\nabla b^2|}{|\nabla b|} + b^{-2} \, \langle \frac{\nabla v}{|\nabla b|} , \nabla |\nabla b^2| \rangle - 2 \, b^{-3} \, v \, 
	\langle \nn , \nabla |\nabla b^2| \rangle \notag \\
	 &=2\,  b^{-1} \, v \,  \frac{ \Delta |\nabla b^2| }{|\nabla b^2|}- b^{-3} \,  a(b) \, v \, \langle \nn , \nabla |\nabla b^2| \rangle - 2 \, b^{-3} \, v \, 
	\langle \nn , \nabla |\nabla b^2| \rangle \, .
\end{align}
Using Proposition \ref{l:sternbochA} and $S \geq 0$,  this becomes
\begin{align}
	\frac{\dv \, V}{|\nabla b|} &\geq \frac{v}{b} \left( -  2\, K+  \frac{ |\bB|^2}{|\nabla b^2|^2} +3 \, \frac{  \Hess_{b^2} (\nn , \nn)  }{  b^2 }
	\right) 
	  - (2 + a(b)) \, b^{-3} \, v \, 
	 \Hess_{b^2} (\nn , \nn) \, .
\end{align}
Thus, since  $v \equiv \frac{1}{A_1 (r)}$ is constant on  $\{ b =r\}$, we see at regular values $r > r_0$ that
\begin{align}
	r\, \int_{b=r} \frac{\dv \, V}{|\nabla b|} &\geq  \int_{b=r} 
	\frac{1}{A_1 (r)} \, \left(
	-  2\, K+  \frac{ |\bB|^2}{|\nabla b^2|^2} + 
	   (1-a(r)) \, 
	\frac{   \Hess_{b^2}(\nn,\nn)}{r^2} \right)
	 \notag \\
	&= \frac{B_2(r)}{A_1 (r)} + (1-a(r)) \, \frac{B_1(r)}{A_1(r)}  
	 - \frac{2}{A_1(r)} \, \int_{b=r} K 
	\, .
\end{align}
Using that $\frac{B_1}{A_1} = 2(a+1)$ and 
 using the connectedness of the level sets and Gauss-Bonnet to bound $ \int_{b=r} K 
 \leq 4\, \pi$, we see that
  \begin{align}
	r\, \int_{b=r} \frac{\dv \, V}{|\nabla b|} &\geq   
	 \frac{B_2(r)}{A_1 (r)} + 2\, (1-a(r)) \, (1+ a(r)) 
	 - \frac{8\, \pi }{A_1(r)} \notag \\
	 &=  \frac{B_2(r)}{A_1 (r)} + 2 - 2\,  a^2(r) 
	 - \frac{8\, \pi }{A_1(r)} 
	 \end{align}
Lemma \ref{c:csCS}
gives that
\begin{align}
	a^2&\leq  \frac{2}{3} \, \frac{B_2}{A_1} \, ,
\end{align}
so we get that
\begin{align}	 
	r\, \int_{b=r} \frac{\dv \, V}{|\nabla b|} &\geq   
	 2  - \frac{1}{2} \, a^2(r)
	 - \frac{8\, \pi }{A_1(r)} 
	\, ,
\end{align}
completely the proof.
 \end{proof}

\begin{proof}[Proof of Proposition \ref{p:aprime}]
The first claim in Lemma \ref{l:defineVforme}, the divergence theorem, and the co-area formula give  that
 \begin{align}
 	a(r_2) - a(r_1)=  \frac{1}{2} \, \int_{r_1 < b < r_2} \dv \, V = \frac{1}{2} \,  \int_{r_1}^{r_2} \int_{b=r} \frac{ \dv \, V}{|\nabla b|} \, dr \, .
 	\label{e:dvplusco}
 \end{align}
The proposition follows from using the second claim in Lemma \ref{l:defineVforme} to get a lower bound for the right-hand side.
\end{proof}

The differential inequality \eqr{e:dia} would force $a$ to rapidly approach  $2$, giving the desired quadratic growth.  We will need to do this in an integral sense because   we only have the differential inequality in an integral sense.  The next corollary  is designed for this.

\begin{Cor}	\label{c:effC}
If $r_1 < r_2$ are regular values of $b$ with $r_0 \leq r_1$ and $a (r) \leq  2$ for $r_1 \leq r \leq r_2$, then 
\begin{align}
	   a(r_2)  \geq   2+ \sqrt{r_1/r_2} \, a(r_1)    - 2\, \sqrt{r_1/r_2}  - \frac{4\, \pi}{\sqrt{r_2}} \, \int_{r_1}^{r_2} \frac{r^{- \frac{1}{2}}}{A_1 (r)} \, dr \, . 
\end{align}
\end{Cor}

\begin{proof}
Let the vector field $V$ be as in Lemma \ref{l:defineVforme}.
Since we work in an interval where $0 < a \leq 2$, we have that 
\begin{align}
	(2 - a^2/2) \geq 2-a
\end{align}
 and, thus, the second claim in Lemma \ref{l:defineVforme}
  gives that
 \begin{align}
 r\, \int_{b=r} \frac{\dv \, V}{|\nabla b|} &\geq   
	 2  -  a(r)
	 - \frac{8\, \pi }{A_1(r)} 
	\, .
\end{align}
Therefore 
since $\dv \, (\sqrt{b} \, V ) = \frac{1}{2 \, \sqrt{b}} \, \langle \nabla b , V \rangle + \sqrt{b} \, \dv \, V$, we see  that
\begin{align}	 \label{e:e220}
	  \int_{b=r} \frac{\dv \, (\sqrt{b} \, V)}{|\nabla b|}  &= \sqrt{r} \, \int_{b=r} \frac{\dv \, V}{|\nabla b|} 
	+ \frac{1}{2} \, r^{ - \frac{1}{2} } \,  \int_{b=r} \langle V , \nn \rangle \notag \\
	& \geq   r^{ - \frac{1}{2}} \, 
	\left( 2 - a(r)
	 - \frac{8\, \pi }{A_1(r)} + \frac{1}{2} \,  \int_{b=r} \langle V , \nn \rangle \right)   \\
	 &= r^{ -  \frac{1}{2} } \, \left( 3   
	 - \frac{8\, \pi }{A_1(r)}   \right)  
	\, . \notag
\end{align}
where the  second equality used \eqr{e:defV}. Thus, \eqr{e:defV}, the divergence theorem,  the co-area formula and \eqr{e:e220} give that
\begin{align}
	2\, \sqrt{r_2} \, (a(r_2)+1) - 2\, \sqrt{r_1} ( a(r_1)+1) & = 
	 \int_{b=r_2} \langle \sqrt{b} \, V , \nn \rangle -  \int_{b=r_1} \langle \sqrt{b} \, V , \nn \rangle = \int_{r_1 \leq b \leq r_2} \dv \, (\sqrt{b} \, V) \notag \\
	 &= \int_{r_1}^{r_2} \frac{ \dv \, (\sqrt{b} \, V)}{|\nabla b|} \, dr  \geq  \int_{r_1}^{r_2}r^{ -  \frac{1}{2} } \, \left( 3   
	 - \frac{8\, \pi }{A_1(r)}   \right)  \, dr \\
	 &= 6\, \sqrt{r_2} - 6\, \sqrt{r_1} - 8\, \pi \, \int_{r_1}^{r_2} \frac{r^{- \frac{1}{2}}}{A_1 (r)} \, dr \, . \notag
\end{align}
Simplifying this gives the claim.
\end{proof}

\subsection{Quadratic growth}

We will show next that $a(r)$ is rapidly approaching $2$.

\begin{Cor}	\label{c:alowb}
There is a constant $c = c (r_0 , \delta) > 0$ so that for all $r \geq r_0$
\begin{align}	\label{e:ar2}
	a(r) \geq 2 - \frac{c}{\sqrt{r}} \, .
\end{align}
\end{Cor}

\begin{proof}
Since $a(r)$ is continuous and the regular values of $b$ are dense, we need only show \eqr{e:ar2} when $r$ is a regular value of $b$.  Moreover, we can assume that  
  $a(r) < 2$ since there is otherwise nothing to prove. 
  
   Suppose now that $r_1 < r_2$ are regular value with
$r_1 \geq r_0$ and so that $a\leq 2$ on $[r_1 , r_2]$.
 Using
 \eqr{e:lowerA},    there is a fixed  $c_1 > 0$, depending on $r_0$ and $\delta$,  so that 
 \begin{align}
 	A_1 (r) \geq c_1 \, r
\end{align}
 and, consequently, 
 \begin{align}
 	\frac{4\, \pi}{\sqrt{r_2}} \, \int_{r_1}^{r_2} \frac{r^{- \frac{1}{2}}}{A_1 (r)} \, dr \leq  \frac{c_2}{\sqrt{r_2}} \, ,
 \end{align}
 where $c_2$ also depends on $r_0 , \delta$.  Thus, the conclusion of Corollary \ref{c:effC} gives that
 \begin{align}	\label{e:fromceff}
	   a(r_2)  \geq   2+ \sqrt{r_1/r_2} \,( a(r_1)    - 2)  - \frac{c_2}{\sqrt{r_2}} \, . 
\end{align}
We will now consider two cases.  

Suppose first that $a < 2$ on $[r_0 , r_2]$, so we can take $r_1 = r_0$ and \eqr{e:fromceff} gives that
\begin{align}	\label{e:fromceffA}
	   a(r_2)  \geq   2+ \sqrt{r_0/r_2} \,( a(r_0)    - 2)  - \frac{c_2}{\sqrt{r_2}} \, .
\end{align}
In particular, \eqr{e:ar2} holds in this case with $c = c_2 - 2 \, \sqrt{r_0}$.

In the remaining case, we have that $\max\,  \{ a(r) \, | \,  r_0 \leq r \leq r_2\} \geq 2$.  Since $a$ is continuous and $a(r_2) < 2$, there exists $r' \in [r_0 , r_2)$ so that
\begin{align}
	a(r') = 2 {\text{ and }} a(r) < 2 {\text{ for }} r' < r \leq r_2 \, .
\end{align}
If $r'$ is a regular value, then applying \eqr{e:fromceff}  with $r_1 = r'$ gives that
 \begin{align}	\label{e:fromceffB}
	   a(r_2)  \geq   2+ \sqrt{r_1/r_2} \,( a(r_1)    - 2)  - \frac{c_2}{\sqrt{r_2}} =  2 - \frac{c_2}{\sqrt{r_2}}\, . 
\end{align}
Thus,  \eqr{e:ar2} holds in this subcase with $c = c_2$.  In the more general case where $r'$ is not a regular value, then we use the density of regular values to choose a sequence of regular values $r_i' > r'$ with $r_i' \to r'$.  Since $a(r)$ is continuous, it follows that
\begin{align}
	\sqrt{r_i'/r_2}\, \, (a(r_i') - 2) \to 0 \, .
\end{align}
Therefore, arguing as above  shows that 
\eqr{e:ar2} holds again with $c = c_2$.
\end{proof}
  
  We can now prove the quadratic lower bound for the growth of $A_1(r)$.
 
\begin{proof}[Proof of Proposition \ref{t:cgrow}]
Since $M$ has one end and zero first betti number,  the level sets of $b$ are connected and their complement consists of a bounded component $p$ and an unbounded component where $b \to \infty$.  Thus, the results of this section apply.

Corollary \ref{c:alowb}
gives constants $c, r_0$ so that for all $r \geq r_0$, we have that $A_1 (r) > 0$ and 
\begin{align}	\label{e:ar2A}
	r\, \frac{A_1'(r)}{A_1(r)} \equiv a(r) \geq 2 - \frac{c}{\sqrt{r}} \, .
\end{align}
Integrating this, we conclude that
\begin{align}
	\log \frac{A_1 (r)}{A_1 (r_0)} = \int_{r_0}^r \frac{a(s)}{s} \, ds \geq \int_{r_0}^r \left( 2\, s^{-1} - c \, s^{ - \frac{3}{2}} \right) \, ds 
	\geq 2 \, \log \frac{r_1}{r_0} - 2 \,c \, r_0^{ - \frac{1}{2}} \, .
\end{align}
Exponentiating this gives that
\begin{align}
	\frac{A_1 (r)}{A_1 (r_0)} \geq \frac{ r^2}{r_0^2} \, \e^{ - 2\, c \, r_0^{ - \frac{1}{2}}} \, ,
\end{align}
which is the claimed quadratic growth of $A_1(r)$.
   \end{proof}

   \subsection{The proof of the gradient estimate}
   
   We will now prove the sharp gradient estimate $A_1(r) \leq 4 \, \pi$ and show that equality (even for just one $r$) implies that $M$ is Euclidean.
   As observed in lemma $3.1$ in \cite{CM4}, this implies that
   \begin{align}
   	\Vol \{ b = r\} \geq 4\, \pi \, r^{2} \, ,
   \end{align}	
   with equality if and only if $M$ is Euclidean.
   
   \begin{proof}[Proof of Theorem \ref{t:main1}]
   We prove first that $A_1  \leq 4\, \pi$.  We will argue by contradiction, so suppose instead that there is some $r_0 > 0$ with $A_1 (r_0) \geq \delta > 0$.
   We can then apply Proposition \ref{t:cgrow} to get for $r \geq r_0$ that
   \begin{align}
   	A_1 (r) \geq c \, r^2 \, , 
   \end{align}
   where $c = c (\delta , r_0) > 0$.
     Since $|\nabla u| = b^{-2} \, |\nabla b|$,  this gives a positive lower bound for the weighted averages of $|\nabla u|$
     \begin{align}
     	r^{-2} \, \int_{b=r} |\nabla u| \, |\nabla b| = r^{-2} \, A_1 (r) \geq c > 0 .
     \end{align}
This 
contradicts that $|\nabla u| \to 0$ at infinity, so we conclude that we must always have $A_1 \leq 4 \, \pi$.

We turn now to the case of equality, where there is some $r_0 > 0$ with $A_1 (r_0) = 4\, \pi$.   Proposition \ref{c:MWa}
gives for any $r$ that
	\begin{align}		\label{e:fromtheprop}
			r\, A_1'(r) &\geq A_1 (r) - 4\, \pi + \frac{1}{2\, r} \, \int_0^r \left( S_1 (s) + B_2 (s)\right)  \, ds  \, .
	\end{align}
We will use this a few times.  First, if we define $f(r) = r^{-1} (A_1(r) - 4\, \pi)$, then \eqr{e:fromtheprop} gives the differential inequality
\begin{align}
	 f' (r) =  r^{-1} \, A_1'(r) - r^{-2} \,   (A_1(r) - 4\, \pi) = r^{-2} \, \left( r \, A_1'(r) - (A_1(r) - 4\, \pi) \right) \geq 0 \, .
\end{align}
Since $f(r_0) = 0$, integrating this gives that 
\begin{align}
	f(r) \geq 0 {\text{ for }} r \geq r_0
\end{align}
 and, thus, $A_1 (r) \geq 4\, \pi$ for $r \geq r_0$.  
	 Since we have already established the opposite inequality $A_1 (r) \leq 4\pi$, we conclude  that
\begin{align}
	A_1(r) = 4\, \pi {\text{ and }} A_1'(r) = 0 {\text{ for all }} r \geq r_0 \, .
\end{align}
Combining this with \eqr{e:fromtheprop}, we see that $B_2(r) = 0$ for all $r$.  Therefore, the trace-free Hessian $B$ of $b^2$ vanishes identically on all of $M$.  By section one in \cite{ChC1}, this implies that $M$ is a cone and, since it is smooth, Euclidean.  
   \end{proof}

 \subsection{An average bound for scalar curvature}

\begin{proof}[Proof of Corollary \ref{t:main1S}]
Theorem \ref{t:main1} gives that $A_1 (2r) \leq 4\, \pi$.  In particular, we have that
\begin{align}
	  \int_r^{2r} A_1'(s) \, ds = A_1 (2r) - A_1(r) \leq 4 \, \pi \, .
\end{align}
Therefore, we can choose $s \in [r,2r]$ with $r \, A_1'(s) \leq 4\, \pi$.  
 Applying
Proposition \ref{c:MWa} with this $s$ 
gives that
	\begin{align}	
     8\, \pi &\geq  2\, r \, A_1'(s) \geq s\, A_1'(s)  \geq A_1 (s) - 4\, \pi + \frac{1}{2\, s} \, \int_0^s \left( S_1 (t) + B_2 (t)\right)  \, dt \notag \\
       &\geq 
       - 4\, \pi + \frac{1}{2\, s} \, \int_0^r   S_1 (t)    \, dt 
       \, .
	\end{align}
Thus, we see that
\begin{align}
	r^{-1} \, \int_0^r \int_{b=s} S    \leq \left( \frac{2s}{r} \right) \, \frac{1}{2s} \, \int_0^r S_1 \leq 4 \, (12\, \pi) = 48 \, \pi  \, .
\end{align}
\end{proof}

     \appendix

 \section{Level sets of harmonic functions}		\label{s:app2}
 
 We will recall a few  facts about the level sets of a proper harmonic function $u$ that are used in the paper.  
 The starting point is that Sard's theorem gives that almost every level set is 
 regular and, thus, is a smooth closed hypersurface. 
 
Section $2$ of \cite{CM1} developed the basics  of weighted averages over level sets of harmonic functions using estimates from \cite{Ch,HS}; cf. 
\cite{CNV,HS,Ln} and proposition $26$ in \cite{ChL}.
In particular,  the integrals of a continuous function over the level sets vary continuously (cf. \cite{CM1},    lemma $11$ in \cite{ChL}):
 
\begin{quote}
 If $f$ is a continuous function on $M \setminus \{ p \}$, then $\int_{b=s} f$ is continuous in $s$.  
\end{quote}

The argument  gives a slightly stronger conclusion.  It is not necessary for $f$ to be defined on the critical set where $|\nabla u| = 0$, as long as $f$ is continuous off of this set and uniformly bounded on compact subsets of $M \setminus \{ p \}$.
  In particular,  the functions $A_1(s), B_1(s)$ and $S_1(s)$ are all continuous in $s$.  
 
 More is true for $A_1(s)$.  Namely, using the divergence theorem and co-area formula, 
 lemma $12$ in \cite{ChL} shows that $A_1(s)$ is locally Lipschitz in $s$.  Thus, by Lebesgue's theorem (the one-dimensional version of Rademacher's theorem), $A_1(s)$ is differentiable almost everywhere and absolutely continuous (i.e., the fundamental theorem of calculus holds).
 
Finally, if $M$ has one end and the first betti number vanishes, then the level sets of $b$ are connected by a result of 
 Munteanu-Wang \cite{MW1} (see Lemma $2.2$ in \cite{MW2}).

\end{document}